\documentclass{article}

\usepackage{geometry}                % See geometry.pdf to learn the layout options. There are lots.
\usepackage{graphicx}

\usepackage{epstopdf}

\usepackage{enumerate}

\usepackage{amsmath}
\usepackage{amssymb}
\usepackage{amsthm}
\usepackage{url}
\usepackage{color}

\usepackage{tikz-cd}

\usepackage{float}
\floatstyle{boxed} 
\restylefloat{figure}
\usepackage[section]{placeins}

\usepackage{chngcntr}
\counterwithin{figure}{section}

\usepackage{pifont}% http://ctan.org/pkg/pifont
\usepackage{stmaryrd}
\usepackage{dsfont}
\usepackage{textcomp}
\usepackage{verbatim}
\usepackage{accents}
\usepackage{wasysym}
\usepackage{csquotes}
\newtheorem{thm}{Theorem}
\newtheorem*{thm*}{Theorem}
\newtheorem{prop}[thm]{Proposition}
\newtheorem{lemma}[thm]{Lemma}

\newtheorem*{mainthm*}{Main Theorem}
\newtheorem*{mainlemma*}{Main Lemma}
\newtheorem{cor}[thm]{Corollary}

\newtheorem*{conj*}{Conjecture}

\theoremstyle{definition}

\newtheorem*{tgtfeat*}{Target features}

\newtheorem*{que*}{Question}

\newtheorem*{aim*}{Aim}

\newtheorem*{diag*}{Diagnosis}
\newtheorem*{mque*}{Main Question}

\newtheorem*{rem*}{Remark}

\newtheorem*{appr*}{Approach}
\newtheorem*{Copernican*}{Copernican Principle}
\newtheorem*{Correctness*}{Correctness Condition}

\newcommand{\Lang}{\mathcal{L}}
\newcommand{\Atom}{\mathrm{Atom}}

\newcommand{\Term}{\mathrm{Term}}

\newcommand{\ev}{\mathrm{ev}}
\newcommand{\Var}{\mathrm{Var}}

\newcommand{\udot}[1]{\text{\d{${#1}$}}}

\newcommand{\gquote}[1]{\ulcorner #1 \urcorner}

\newcommand{\sbt}{\mathrm{sbt}}
\newcommand{\gq}{\mathrm{gq}}
\newcommand{\CK}{\mathrm{CK}}

\newcommand{\df}{\mathrm{df}}

\newcommand{\Tr}{\mathrm{T}}

\newcommand{\Ag}{\mathrm{Ag}}

\newcommand{\Kn}{\mathrm{K}}

\newcommand{\dt}{\hspace{2pt}}

\newcommand{\PA}{\mathsf{PA}}

\newcommand{\UCT}{\mathsf{UCT}}
\newcommand{\FS}{\mathsf{FS}}

\newcommand{\Base}{\mathsf{Base}}

\newcommand{\KT}{\mathsf{KT}}
\newcommand{\DCB}{\mathsf{DCB}}

\newcommand{\NEC}{\mathsf{NEC}}
\newcommand{\CONEC}{\mathsf{CONEC}}

\usepackage[square,numbers]{natbib}
\usepackage[displaymath, mathlines, pagewise]{lineno}
\title{A genuinely untyped solution to the knower paradoxes}

\author{Paul Gorbow}

\date{}                                           % 

\begin{document}

\maketitle

\section{Introduction}

In \cite{KM60}, Kaplan and Montague showed that certain intuitive axioms for a first-order theory of knowledge, formalized as a predicate, are jointly inconsistent. A similar result was established by Montague in \cite{Mon63}. Both arguments rely on self-referential formulas. I offer a consistent first-order theory solving these {\em knower paradoxes}, with the following main features:
\begin{itemize}
\item It solves the knower paradoxes by providing a faithful formalization of the principle of veracity (that knowledge requires truth), using both a knowledge and a truth predicate.
\item It is genuinely untyped. I.e. it is untyped not only in the sense that it uses a single knowledge predicate applying to all sentences in the language (including sentences in which this predicate occurs), but in the sense that its axioms quantify over all sentences in the language, thus supporting comprehensive reasoning with untyped knowledge ascriptions.
\item Common knowledge predicates can be defined in the system using self-reference. This fact, together with the genuinely untyped nature of the system and a technique based on L\"ob's theorem, enables it to support comprehensive reasoning with untyped common knowledge ascriptions (without having any axiom directly addressing common knowledge).
\end{itemize}

Let $L$ be the language of Peano Arithmetic ($\PA$), and let $L_\mathrm{K}$ be $L$ augmented with a predicate $\mathrm{K}$ for knowledge. Utilizing G\"odel's implementation of syntax in arithmetic (from \cite{God31}), any formula $\phi$ in $L_\mathrm{K}$ is represented by a term $\gquote{\phi}$ in $L$. Suppose that $\phi$ is a sentence. Now ``the agent knows $\phi$'', ``we know $\phi$'', or suchlike, can be formalized as $\mathrm{K}(\gquote{\phi})$. Moreover, there is a formula $\Pr_\PA$ in $L$, such that ``$\phi$ is provable in $\PA$'' is represented by $\Pr_\PA(\gquote{\phi})$. In this context, the above results of Kaplan and Montague can be stated as follows:

\begin{thm*}[Kaplan \& Montague, 1960]
The $L_\mathrm{K}$-theory extending Peano Arithmetic with the following axiom schemata is inconsistent, where $\phi, \psi$ range over sentences in $L_\mathrm{K}$:
\begin{align*}
& \mathrm{K}(\gquote{\phi}) \rightarrow \phi \tag{$\mathsf{UT}^\Kn$} \\
& \mathrm{K}(\gquote{\mathrm{K}(\gquote{\phi}) \rightarrow \phi}) \tag{$\Kn[\mathsf{UT}^\Kn]$} \\
& {\Pr}_\PA(\gquote{\phi \rightarrow \psi}) \wedge \mathrm{K}(\gquote{\phi}) \rightarrow \mathrm{K}(\gquote{\psi}) \tag{$\mathsf{I}^\Kn$}
\end{align*}
\end{thm*}

\begin{thm*}[Montague, 1963]
The $L_\mathrm{K}$-theory extending Peano Arithmetic with the following axiom schema and deductive rule is inconsistent, where $\phi$ ranges over sentences in $L_\mathrm{K}$:
\begin{align*}
& \mathrm{K}(\gquote{\phi}) \rightarrow \phi \tag{$\mathsf{UT}^\Kn$} \\
& \vdash \phi \phantom{~~} \Longrightarrow \phantom{~~} \vdash \mathrm{K}(\gquote{\phi}) \tag{$\NEC^\mathrm{K}$}
\end{align*}
\end{thm*}

$\mathsf{UT}^\Kn$ formalizes the classical philosophical principle of veracity (that knowledge requires truth) and stands for {\em Untyped T axiom for $\Kn$} (as it is an untyped analogue of the axiom $\mathsf{T}$ from modal logic); $\NEC^\Kn$ formalizes that deduction in the system is a legitimate method for generating knowledge and stands for {\em Necessitation for $\Kn$}; $\mathsf{I}^\Kn$ formalizes that knowledge is closed under deduction in $\PA$ and derives from notation in \cite{KM60}. Note that $\Kn[\mathsf{UT}^\Kn]$ follows from a single application of $\NEC^\Kn$ to $\mathsf{UT}^\Kn$. The proofs of these theorems rely on G\"odel's fixed-point lemma. The former uses a sentence $\kappa$, such that $\kappa \leftrightarrow \Kn(\neg \kappa)$, while the latter uses a sentence $\delta$, such that $\delta \leftrightarrow \neg \Kn(\delta)$. 

\S \ref{Sec:Prel} sets the notation used throughout the paper and specifies the language and basic assumptions on the theories considered. It also briefly explains G\"odel's implementation of syntax.

In \S \ref{Sec:Theory_KT}, I list several target features for an adequate system of knowledge and truth---a design specification, if you will---consisting of the main features above along with several other desirable features. I then introduce the first-order axiomatic system $\KT$ of knowledge and truth, and show that it has all these features. In particular, $\KT$ formalizes veracity as an axiom expressing that for all sentences $s$ in the language, $\Kn(s) \rightarrow \Tr(s)$, where $\Tr$ is a truth predicate. As it has further axioms and rules from which $\NEC^\Kn$ and $\mathsf{I}^\Kn$ can be derived, and is consistent (as explained below), it is a solution to the knower paradoxes.

A key design choice, motivated by the desire to obtain a genuinely untyped system, is that $\KT$ is axiomatized as an extension of the Friedman--Sheard system of truth ($\FS$, introduced in \cite{FS87}). Halbach showed in \cite{Hal94} that $\FS$ is locally modeled by the revision semantics of Gupta and Herzberger (from \cite{Gup82, Her82a, Her82b}), in the sense that every finite fragment of $\FS$ is satisfied by an expansion of the standard model of arithmetic obtained by a finite number of such revisions of the interpretation of the truth predicate. In \citep{Ste14}, Stern developed a version of the revision semantics for necessity, treated as a first-order predicate rather than as a modal operator, which generalizes the usual possible worlds semantics. In \S \ref{Sec:Semantics}, I utilize Stern's technique to develop a revision semantics for multi-agent knowledge. This approach generalizes Hintikka's possible worlds approach from \cite{Hin62}, and is thus coherent with that tradition. The consistency of $\KT$ comes out as a corollary of Theorem \ref{Thm:KT_semantics}, which shows that a certain extension of $\KT$ is locally modeled by this revision semantics.

$\KT$ has a self-referential axiom $\mathsf{R}_\DCB$, asserting that knowledge is deductively closed in a robust sense; it is defined as a G\"odel fixed point. On top of its intuitive appeal, the value of this axiom lies in that it harmonizes both with the revision semantics and with an interesting proof-technique based on L\"ob's theorem. The latter enables me to show that $\KT$ supports expedient reasoning about common knowledge defined as a G\"odel fixed-point (without assuming any principles beyond ordinary knowledge). In particular, I show that the defined common knowledge predicate is unique and deductively closed, in certain precise senses. These results are explained in \S \ref{Sec:DCB_CK_well-behaved}.

In summary, the over-arching theme is that the knower paradoxes are solved in a genuinely untyped manner; particular innovations include the faithful formalization of veracity, the generalization of the revision semantics to multi-agent knowledge, the robust formalization of deductive closure as a G\"odel fixed-point along with a proof that this can be validated in the revision semantics, and the definition of common knowledge as a G\"odel fixed-point along with proofs that this defined predicate is unique and deductively closed.

\section{Formal notation and preliminaries}\label{Sec:Prel}

Let $L$ be a first-order language. An {\em $L$-system} is defined as a recursively enumerable set of $L$-sentences that includes (a Hilbert-style) axiomatization of first-order logic and is closed under modus ponens. Each systems is given as a list of axioms and deductive rules, such as the deductive rule $\NEC^\Kn$ above. Note that if a system $S$ with a deductive rule, say $\NEC^\Kn$, is extended to a system $S'$, then $\NEC^\Kn$ can be applied in $S'$ to any theorem of $S'$. On the other hand, $\NEC^\Kn$ can generally not be applied in $S$ to theorems of $S'$; in other words, it cannot be applied to undischarged assumptions, so it cannot in general be used to derive $\phi \rightarrow \Kn(\gquote{\phi})$, for a sentence $\phi$. I employ the following notation, for any arity $n \in \mathbb{N}$:
\begin{align*}
\Var &=_\df \textnormal{the set of variables.} \\
L^n &=_\df \textnormal{the set of $L$-formulas with precisely $n$ free variables.} \\
L^n[v_1, \cdots, v_n] &=_\df \textnormal{the set of $L$-formulas with precisely $v_1, \cdots, v_n$ as free variables.} \\
\Atom_L^n &=_\df \textnormal{the set of atomic $L$-formulas with precisely $n$ free variables.} \\
\Term_L^n &=_\df \textnormal{the set of $L$-terms with precisely $n$ free variables.} 
\end{align*}
When the superscript $n$ above is dropped, all arities are included, e.g. $\Term_L$ is the set of all $L$-terms. The {\em substitution function}, denoted $\sbt : L \times \Term_\Lang \times \Var \rightarrow \Lang$, maps $\langle \phi, t, v \rangle$ to the formula obtained from $\phi$ by substituting every instance of the variable $v$ by the term $t$.

Let $\Lang$ be the first-order language with the usual symbols $\underline{0}, S, +, \cdot$ of arithmetic, a unary predicate symbol $U$, a unary function symbol $u$, a unary predicate symbol $\Kn^1$, a binary relation symbol $\Kn^2$, a unary predicate symbol $\Tr$, and a unary predicate symbol $\Ag$. When the arity is clear from the context, $\Kn^1$ and $\Kn^2$ will both be denoted $\Kn$. The sublanguage of $\Lang$ generated by the arithmetical symbols is denoted $\Lang_\PA$ and the sublanguage generated by $\underline{0}, S, \cdot, \times, U, u, \Ag$ is denoted $\Lang_-$. All the systems of the paper are formulated as extensions of Peano Arithmetic ($\PA$) in $\Lang$, or a sublanguage thereof. The standard $\Lang_\PA$-model of arithmetic is denoted $\mathbb{N}$.

$\Kn(x, y)$ is intended to express ``the agent $x$ knows the sentence $y$''. $\Kn(y)$ is intended to express ``we know $y$'', or suchlike. $\Tr(y)$ is intended to express ``$y$ is true''. $\Ag(x)$ is intended to express ``$x$ is an agent''. Utilizing the method of coding in arithmetic, the symbols $U, u$ are employed as universal relation and function symbols, available for the sake of generality, ensuring that the systems considered can be seamlessly extended to apply to any recursively enumerable language.\footnote{In the context of arithmetic, if for each $n, k \in \mathbb{N}$, $R^n_k$ is an $n$-ary relation symbol, then each $R^n_k(x_1, \cdots, x_n)$ can be interpreted by $U(\langle S^n(\underline{0}), S^k(\underline{0}), x_1, \cdots, x_n \rangle)$. To be precise, the latter abbreviates $\exists y \dt \big( U(y) \wedge \phi(S^n(\underline{0}), S^k(\underline{0}), x_1, \cdots, x_n, y) \big)$ for a formula $\phi(S^n(\underline{0}), S^k(\underline{0}), x_1, \cdots, x_n, y) \big)$ expressing that $y$ codes the $n$-tuple $\langle S^n(\underline{0}), S^k(\underline{0}), x_1, \cdots, x_n \rangle$.}

As is familiar, G\"odel showed in \cite{God31} how to implement syntactical resources in arithmetic. Here follows some notation and basic properties of this machinery: For any $n \in \mathbb{N}$, $\mathrm{num}(n)$ (or $\underline{n}$) denotes the {\em numeral} $S^n(\underline{0}) \in \Term_{\Lang_\PA}^0$. Let $t \in \Term_\Lang$ and $\phi \in \Lang$. The {\em G\"odel quotes} $\gq(t)$ (or $\gquote{t}$) and $\gq(\phi)$ (or $\gquote{\phi}$) denote numerals which represent $t$ and $\phi$, respectively, in arithmetic. For example, if $\phi \in \Lang^0$, then $\Kn(x, \gquote{\phi})$ expresses ``$x$ knows `$\phi${'}{''}.  Closely related to these are the {\em G\"odel codes} $\mathrm{gc}(t), \mathrm{gc}(\phi)$ of $t$ and $\phi$, respectively, which are the numbers in the standard model interpreting $\gq(t), \gq(\phi)$, respectively. 

In order to conveniently utilize G\"odel's machinery, the systems considered in this paper are formulated as extensions of a {\em base system}, $\Base$. One can get by with Robinson's arithmetic ($\textsf{Q}$) as base system for syntactic purposes. But since proof by induction is so natural, I require that $\Base$ proves all the theorems of $\PA(\Lang)$, which I axiomatize as $\PA$ with its induction schema extended to $\Lang$. I also require that $\Base$ has axioms for the various representations involved in the G\"odel machinery, which I now proceed to explain:

Let $\phi \in \Lang^n$, let $R(x_1, \cdots, x_n)$ be a relation on $\mathbb{N}$ and let $S$ be an $\Lang$-system. $\phi$ {\em represents} $R$ {\em in} $S$ if for all $a_1, \cdots, a_n \in \mathbb{N}$,
\begin{align*}
R(a_1, \cdots, a_n) \Longleftrightarrow S \vdash \phi(\underline{a_1}, \cdots, \underline{a_n}).
\end{align*}
Representation of functions is analogously defined. $\Base$ is assumed to be equipped with various function and relation symbols, representing various syntactically relevant functions and relations. Formally, these symbols are implemented by means of $U$ and $u$ as explained above, and by means of axioms of $\Base$ asserting what relation or function each such symbol represents. These representations are denoted with a dot under the name of the relation/function in question. For example, $\udot\Lang^0$ denotes the representation of $\Lang^0$, and $\udot \wedge$ represents a function from $\Lang \times \Lang$ to $\Lang$ mapping two formulas to their conjunction. The evaluation function returning the numeric value of any closed $\Lang_\PA$-term is denoted $\ev : \Term_{\Lang_\PA}^0 \rightarrow \mathbb{N}$. Its representation $\udot\ev(x)$ is also denoted $x^\circ$.

Since $\Base \vdash \PA$, the G\"odel--Carnap fixed-point lemma is available in $\Base$, facilitating formalization of self-reference. More precisely, for any $\phi(x_1, \cdots, x_{n+1}) \in \Lang^{n+1}$ there is a formula $\theta(x_1, \cdots, x_n) \in \Lang^n$, such that
\begin{align*}
\Base \vdash \forall x_1, \cdots, x_n \big( \theta(x_1, \cdots, x_n) \leftrightarrow \phi(x_1, \cdots, x_n, \gquote{\theta}) \big).
\end{align*}

The G\"odel machinery enables us to formalize statements about knowledge, for example the statement that the knowledge of each agent is closed under modus ponens:
\begin{align*}
\forall \alpha \forall \phi \forall \psi \Big( \big( \Ag(\alpha) \wedge \udot\Lang^0(\phi) \wedge \udot\Lang^0(\psi) \big) \rightarrow \big( \Kn(\alpha, \phi \udot\rightarrow \psi) \rightarrow (\Kn(\alpha, \phi) \rightarrow \Kn(\alpha, \psi)) \big) \Big)
\end{align*}
In order to make such statements more readable, I introduce the following abbreviations, for any formula $\phi$ in one free variable, and any formula $\psi$:
\begin{align*}
\forall x \in \phi \dt \big( \psi(x) \big) &\textnormal{ abbreviates } \forall x \big(\phi(x) \rightarrow \psi(x)\big) \\
\exists x \in \phi \dt \big( \psi(x) \big) &\textnormal{ abbreviates } \exists x \big(\phi(x) \wedge \psi(x)\big) \\
\forall x_1, \cdots x_n \psi  &\textnormal{ abbreviates } \forall x_1 \cdots \forall x_n \psi \\
\exists x_1, \cdots x_n \psi  &\textnormal{ abbreviates } \exists x_1 \cdots \exists x_n \psi 
\end{align*}
These abbreviations may also be combined. For example, the above sentence is written as follows:
\begin{align*}
\forall \alpha \in \Ag \dt \forall \phi, \psi \in \udot\Lang^0 \dt \big( \Kn(\alpha, \phi \udot\rightarrow \psi) \rightarrow (\Kn(\alpha, \phi) \rightarrow \Kn(\alpha, \psi)) \big)
\end{align*}

\section{A system of knowledge and truth}\label{Sec:Theory_KT}

This section introduces and motivates the $\Lang$-system $\KT$ of knowledge and truth, intended to solve the knower paradoxes. There are many possibilities for doing so, and it is therefore helpful to take a step back to consider what features the system as a whole ought to have. 

\begin{tgtfeat*}
I consider the following features especially significant:
\begin{enumerate}
\item\label{tgtfeat:paradox} The system solves the knower paradoxes in that it consistently combines natural formalizations of the philosophical principles of veracity, deductive closure and necessitation in a genuinely untyped framework. \\
{\em Motivation:} It is of fundamental philosophical importance to show that we can safely reason with such an intuitive concept of knowledge in a setting with the self-referential capabilities of natural language. The various components of this target feature are subsumed in the specific Target features \ref{tgtfeat:self-ref}--\ref{tgtfeat:semantics} below. 
\item\label{tgtfeat:self-ref} The system is genuinely untyped, in the sense that it supports reasoning with knowledge ascriptions for arbitrary $\Lang$-formulas (including self-referential ones). \\
{\em Motivation:} Detailed arguments for the legitimacy of self-reference are given by Kripke in \cite{Kri75}, for example:
\begin{displayquote}
{[}G\"odel{]} also showed that elementary syntax can be interpreted in number theory. In this way, G\"odel put the issue of the legitimacy of self-referential sentences beyond doubt; he showed that they are as incontestably legitimate as arithmetic itself.
\end{displayquote}
Moreover self-reference adds quite usefully to the expressive power of the language; indeed it is utilized to formalize deductive closure and to define common knowledge in this paper.
\item\label{tgtfeat:veracity} The system faithfully formalizes the principle of veracity. \\
{\em Motivation:} It is a basic requirement in the classical concept of knowledge that for something to be known it must be true.
\item\label{tgtfeat:NEC} The system supports necessitation for knowledge ($\NEC^\Kn$). \\
{\em Motivation:} $\NEC^\Kn$ formalizes the natural assumption that ones axioms are known and that deductive reasoning from these axioms is a legitimate method for generating knowledge. This principle does not entail that the system proves an internally quantified assertion that everything provable in the system is known (assuming that $\neg \Kn(\gquote{\bot})$ is derivable, such an assertion would contradict G\"odel's second incompleteness theorem). It merely formalizes the principle that whatever theorem is derived in the system is passed on as knowledge to all agents. Even though this may be regarded an idealization, it would be methodologically unacceptable if $\NEC^\Kn$ were inconsistent with other well-motivated principles.
\item\label{tgtfeat:DC} The system proves natural principles of deductive closure for both truth and knowledge.
{\em Motivation:} Consider $\mathsf{I}^\Kn$ from the paradox of Kaplan and Montague, which entails an internally quantified assertion that every theorem of $\PA$ is known. Although it is natural to assume that all those theorems are true, the concept of knowledge in natural language hardly supports the claim that they are all known. Even so, there are reasons for including deductively closed knowledge among the target features: In order to solve the paradoxes, this principle needs to be shown to be consistent. Moreover, even if knowledge is not taken to be deductively closed, the concept obtained as the deductive closure of knowledge is of significant interest, because it serves as a limiting idealization approximating knowledge and provides a useful connection to logic.
\item\label{tgtfeat:semantics} The system has a well-motivated semantics. \\
{\em Motivation:} A semantics shows that the system is consistent, but it also gives insight on the meaning of the axioms, and what conception of truth and knowledge the system axiomatizes.
\item\label{tgtfeat:agents} The system supports reasoning about multiple agents, who share sufficient knowledge about the syntax. \\
{\em Motivation:} Different agents may have different knowledge. But in order to sensibly compare their knowledge, it is reasonable to assume that they all have sufficient knowledge about the syntax, e.g. that they all know ``$\gquote{\phi \wedge \psi}$ is the conjunction of $\gquote{\phi}$ and $\gquote{\psi}$''. 
\item\label{tgtfeat:sym} For any sentence $\phi \in \Lang$, the system proves $\phi$ iff it proves $\Tr(\gquote{\phi})$. \\
{\em Motivation:} In the literature, such a system is called {\em symmetric}. This condition ensures that the formalization of truth is unambiguous, in the following sense: There is a special consideration that comes into play when formalizing truth in a system $S$, namely that for any $\phi \in \Lang$, if $S \vdash \phi$ then $\phi$ is naturally considered true (say {\em externally true}), while if $S \vdash \Tr(\gquote{\phi})$ then $\phi$ is also considered true (say {\em internally true}). Symmetry is given by the following deductive rules: 
\begin{align*}
\vdash \phi \phantom{~~} \Longrightarrow \phantom{~~} \vdash \Tr(\gquote{\phi}) \tag{$\NEC^\Tr$} \\
\vdash \phi \phantom{~~} \Longleftarrow \phantom{~~} \vdash \Tr(\gquote{\phi}) \tag{$\CONEC^\Tr$}
\end{align*}
Note that symmetry may be viewed as a weakening of the following schema of {\em Tarski biconditionals}, where $\phi$ ranges over sentences in $\Lang$:
\begin{align*}
\Tr(\gquote{\phi}) \leftrightarrow \phi \tag{$\textsf{TB}$}
\end{align*}
By the well-known liar paradox, this schema is inconsistent over $\PA$.
\item\label{tgtfeat:UT} The system proves a philosophically satisfying restriction of $\mathsf{TB}$, and explains why these instances hold. \\
{\em Motivation:} Since $\mathsf{TB}$ has many natural instances, it is desirable that the system proves a philosophically satisfying restriction of $\mathsf{TB}$ from well-motivated axioms, thus also explaining why they hold.\footnote{An observation of McGee in \cite{McG92} shows that for every $\Lang_\PA$-sentence $\psi$, there is an instance $\theta$ of $\Tr(\gquote{\phi}) \leftrightarrow \phi$, such that $\Base \vdash \psi \leftrightarrow \theta$. So in light of G\"odel's incompleteness theorem we cannot hope for an axiomatization that proves every true instance of $\Tr(\gquote{\phi}) \leftrightarrow \phi$, just as no axiomatization of arithmetic can prove every true arithmetic sentence.}
\item\label{tgtfeat:Introspec} The system is consistent with supplementary axioms of knowledge, motivated by philosophical applications. \\
{\em Motivation:} Depending on ones philosophical view or purposes, one may wish to adopt such supplementary axioms.
\item\label{tgtfeat:CK} The system defines and supports reasoning with common knowledge, without formalizing any principle directly addressing common knowledge. \\
{\em Motivation:} Such a system has theoretical simplicity, in that it explains common knowledge in terms of knowledge, establishing that there is no need for a separate treatment of common knowledge.
\end{enumerate}
\end{tgtfeat*}

In order to meet the Target features, I have designed the system $\KT$ ({\em Knowledge and Truth}, axiomatized in Figure \ref{Fig:KT}) as an extension of $\FS$ (the {\em Friedman--Sheard system of truth}, axiomatized in Figure \ref{Fig:FS}) and as an extension of $\DCB$ ({\em Deductively Closed Belief}, axiomatized in Figure \ref{Fig:DCK}). $\FS$ was introduced in \cite{FS87}, while the present axiomatization is from \cite{Hal94}.  A self-contained exposition of central results concerning $\FS $ is found in \cite[ch. 14]{Hal14}. The system $\mathsf{UCT}$ ({\em Untyped Compositional Truth}) is obtained from $\FS$ by removing the deductive rules $\NEC^\Tr$ and $\CONEC^\Tr$. $\DCB$ is an innovation of this paper, which is intended as a significant system in its own right providing a basic and robust formalization of deductively closed belief for multiple agents. $\Kn$ is intended as a belief relation or as a knowledge relation according to whether $\DCB$ is considered as a stand-alone system or as a subsystem of $\KT$, respectively.

\begin{figure}

\caption{Axioms of $\DCB$}
\label{Fig:DCK}

\begin{center}
 
\vspace{12pt}
{\bf The system $\DCB$}

\vspace{12pt}
$\DCB$ extends $\Base$ with the following axioms:

\end{center}
\begin{align*}
\exists \alpha &\dt \big( \Ag(\alpha) \big) \tag{$\textsf{Non-triviality}$} \\
\forall \phi \in \udot\Lang^{{0}} &\dt \big(\Kn^1(\phi) \leftrightarrow \forall \alpha \in \Ag \dt \Kn^2(\alpha, \phi) \big) \tag{$\Kn^1$-$\Kn^2$} \\
\forall \alpha \in \Ag \dt \forall \phi, \psi \in \udot \Lang^{{0}} &\dt \Big( \big( \Kn(\alpha, \phi) \wedge \Kn(\alpha, \phi \udot\rightarrow \psi) \big)  \rightarrow \Kn(\alpha, \psi) \Big) \tag{$\textsf{UK}^\Kn$} \\
\forall \alpha \in \Ag \dt \forall \phi \in \udot \Lang^{{0}} &\dt \big( {\Pr}_{\DCB}(\phi)  \rightarrow \Kn(\alpha, \phi) \big) \tag{$\mathsf{R}_\DCB$}
\end{align*}

\end{figure}

\begin{figure}

\caption{Axioms and rules of $\FS $}
\label{Fig:FS}

\begin{center}
 
\vspace{12pt}
{\bf The system $\FS $}

\vspace{12pt}
$\FS $ extends $\Base$ with the following axioms and rules:

\vspace{18pt}
{\bf Axioms of $\FS $}
\end{center}
\begin{align*}
\forall t_1, \cdots, t_n \in \udot \Term_{\udot \Lang_\PA}^{{0}} & \dt \big( \Tr \big( \udot R(t_1, \cdots, t_n) \big)  \leftrightarrow R(t_1^\circ, \cdots, t_n^\circ) \Big), \textnormal{for each $R(\vec{x}) \in \Atom_{\Lang_-}^n$.} \tag{$\UCT^\Atom$}  \\
 \forall \phi \in \udot \Lang^{{0}} & \dt \big( \Tr(\udot \neg \phi)   \leftrightarrow \neg \Tr(\phi) \big) \tag{$\UCT^\neg$} \\
 \forall \phi, \psi \in \udot \Lang^{{0}} & \dt \big( \Tr(\phi \udot\rightarrow \psi)   \leftrightarrow (\Tr(\phi) \rightarrow \Tr(\psi)) \big) \tag{$\UCT^\rightarrow$} \\
 \forall \phi, \psi \in \udot \Lang^{{0}} & \dt \big( \Tr(\phi \udot\wedge \psi)   \leftrightarrow (\Tr(\phi) \wedge \Tr(\psi)) \big) \tag{$\UCT^\wedge$} \\
 \forall \phi, \psi \in \udot \Lang^{{0}} & \dt \big( \Tr(\phi \udot\vee \psi)   \leftrightarrow (\Tr(\phi) \vee \Tr(\psi)) \big) \tag{$\UCT^\vee$} \\
 \forall v \in \udot \Var \dt \forall \phi \in \udot \Lang^{{1}}[v] & \dt \big( \Tr(\udot \forall v \phi)   \leftrightarrow \forall t \in \udot\Term_{\udot \Lang}^0 \big( \Tr(\udot\sbt(\phi, t, v)) \big) \Big) \tag{$\UCT^\forall$} \\
 \forall v \in \udot \Var \dt \forall \phi \in \udot \Lang^{{1}}[v] & \dt \big( \Tr(\udot \exists v \phi)   \leftrightarrow \exists t \in \udot\Term_{\udot \Lang}^0 \big( \Tr(\udot\sbt(\phi, t, v)) \big) \Big)  \tag{$\UCT^\exists$} 
\end{align*}

\vspace{6pt}
\begin{center} 
{\bf Rules of $\FS $}
\end{center}
\begin{align*}
\vdash \phi \phantom{~~} &\Longrightarrow \phantom{~~} \vdash \Tr(\gquote{\phi}) \textrm{, for each $\phi \in \Lang^0$.}  \tag{$\NEC^\Tr$} \\
\vdash \phi \phantom{~~} &\Longleftarrow \phantom{~~} \vdash \Tr(\gquote{\phi}) \textrm{, for each $\phi \in \Lang^0$.}  \tag{$\CONEC^\Tr$}
\end{align*}

\end{figure}

\begin{figure}
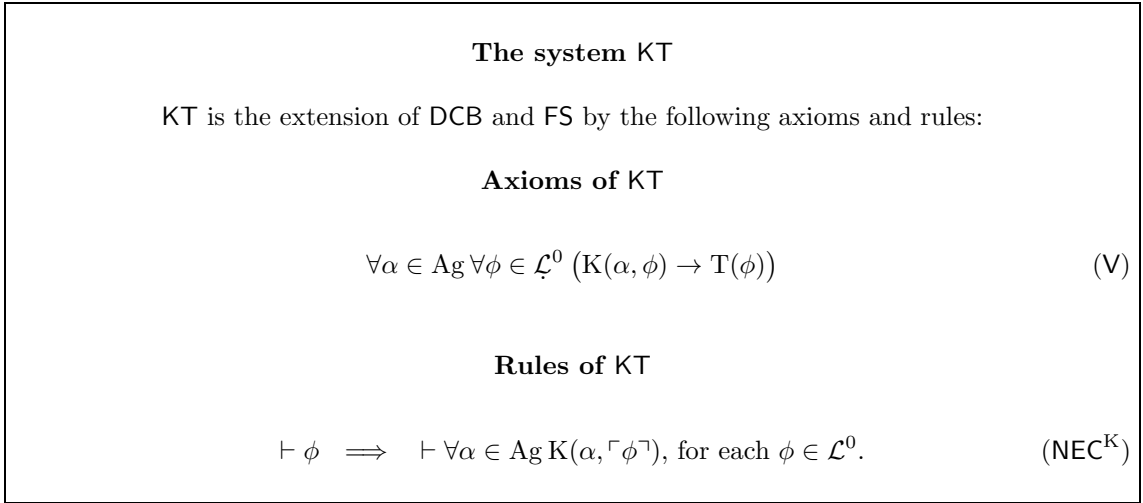


\caption{Axioms and rules of $\KT $}
\label{Fig:KT}

\begin{center}
 \vspace{12pt}
{\bf The system $\KT $}

\vspace{12pt}
$\KT$ is the extension of $\DCB$ and $\FS$ by the following axioms and rules:

\vspace{12pt}
{\bf Axioms of $\KT $}
\end{center}
\begin{align*}
\forall \alpha \in \Ag \dt \forall \phi \in \udot \Lang^{{0}} &\dt \big( \Kn(\alpha, \phi)  \rightarrow \Tr(\phi) \big) \tag{$\mathsf{V}$}
\end{align*}

\vspace{6pt}
\begin{center} 
{\bf Rules of $\KT $}
\end{center}
\begin{align*}
\vdash \phi \phantom{~~} &\Longrightarrow \phantom{~~} \vdash \forall \alpha \in \Ag \dt \Kn(\alpha, \gquote{\phi}) \textrm{, for each $\phi \in \Lang^0$.}  \tag{$\NEC^\Kn$}
\end{align*}

\end{figure}

\begin{figure}
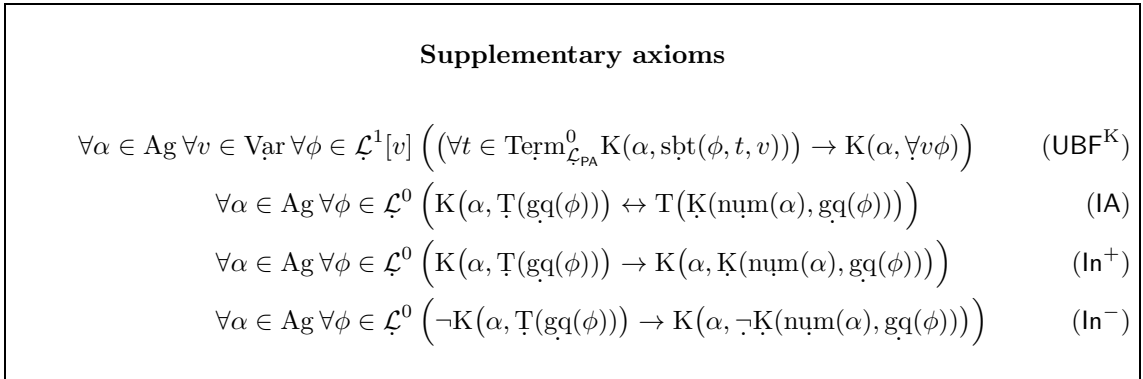


\caption{Supplementary axioms}
\label{Fig:Sup}

\begin{center}
 \vspace{12pt}
{\bf Supplementary axioms}
\end{center}
\begin{align*}
\forall \alpha \in \Ag \dt \forall v \in \udot\Var \dt \forall \phi \in \udot\Lang^1[v] &\dt \Big( \big( \forall t \in \udot\Term_{\udot\Lang_\PA}^0 \Kn(\alpha, \udot\sbt(\phi, t, v)) \big) \rightarrow \Kn(\alpha, \udot\forall v \phi) \Big) \tag{$\mathsf{UBF}^\Kn$} \\
\forall \alpha \in \Ag \dt \forall \phi \in \udot\Lang^0 &\dt \Big( \Kn\big(\alpha, \udot\Tr(\udot\gq(\phi)) \big) \leftrightarrow \Tr\big( \udot\Kn(\udot{\mathrm{num}}(\alpha), \udot\gq(\phi)) \big) \Big) \tag{$\mathsf{IA}$} \\
\forall \alpha \in \Ag \dt \forall \phi \in \udot\Lang^0 &\dt \Big( \Kn\big(\alpha, \udot\Tr(\udot\gq(\phi)) \big) \rightarrow \Kn\big( \alpha, \udot\Kn(\udot{\mathrm{num}}(\alpha), \udot\gq(\phi)) \big) \Big) \tag{$\mathsf{In}^+$} \\
\forall \alpha \in \Ag \dt \forall \phi \in \udot\Lang^0 &\dt \Big( \neg \Kn\big(\alpha, \udot\Tr(\udot\gq(\phi)) \big) \rightarrow \Kn\big( \alpha, \udot\neg \udot\Kn(\udot{\mathrm{num}}(\alpha), \udot\gq(\phi)) \big) \Big) \tag{$\mathsf{In}^-$}
\end{align*}

\end{figure}

I proceed to show how $\KT$ fulfills the respective Target features \ref{tgtfeat:paradox}--\ref{tgtfeat:CK} and to explain the axioms of $\KT$.

\begin{enumerate}
\item This Target feature is subsumed in Target features \ref{tgtfeat:self-ref}--\ref{tgtfeat:semantics}.
\item A common approach to curb the role self-reference plays in the paradoxes, while allowing self-reference in the language, is to introduce a typing-regiment. One may e.g. introduce a hierarchy of symbols $\Tr_i$ and $\Kn_i$, for all $i < \omega$, and formulate the axioms about each $\Tr_i, \Kn_i$ so that they only apply to formulas not including $\Tr_j, \Kn_j$ for all $j \geq i$. However, this is hard to motivate philosophically and it cripples reasoning with useful self-referential formulas such as defined common knowledge predicates. So in order to meet Target feature \ref{tgtfeat:self-ref}, the system is designed to be genuinely untyped. This is achieved in that all relevant axioms quantify universally over $\udot\Lang^0$ (or $\udot\Lang^1$).
\item The axiom $\mathsf{V}$ of $\KT$ is a straight-forward faithful formalization of the principle of veracity, employing formal symbols for both the knowledge relation and the truth predicate.
\item $\KT$ has the rule $\NEC^\Kn$.
\item The axiom $\mathsf{UK}^\Kn$ of $\KT$ asserts that knowledge is closed under modus ponens. It is easily observed that this axiom, together with the rule $\NEC^\Kn$, ensures that for any $\phi, \psi \in \Lang$, if $\KT \vdash \Kn(\gquote{\phi})$ and $\KT \vdash \phi \rightarrow \psi$, then $\KT \vdash \Kn(\gquote{\psi})$. So in this sense, $\KT$ formalizes the intuition that knowledge is closed under any (as externally quantified) reasoning sanctioned by $\KT$. Moreover, the axiom $\mathsf{R}_\DCB$ of {\em Reflection over $\DCB$} is an innovation of this paper that provides additional internally quantified deductive closure. In particular, it is easily seen to entail the schema $\mathsf{I}^\Kn$ assumed in the paradox of Kaplan and Montague. $\mathsf{R}_\DCB$ is a robust way of asserting that $\Kn$ is a deductively closed extension of the basic provability predicate $\Pr_\DCB$ of $\DCB$. This robustness is a consequence of that $\mathsf{R}_\DCB$ is itself an axiom of the system $\DCB$, whose provability predicate it refers to. Therefore, $\mathsf{R}_\DCB$ actually formalizes a self-referential sentence, and its existence follows from the G\"odel fixed-point lemma. Note that it is not possible to reformulate $\KT$ by substituting $\forall \phi \in \udot{\Lang^0} \dt \big( {\Pr}_{\KT }(\phi) \rightarrow \Kn(\phi) \big)$ for $\mathsf{R}_\DCB$, because by $\mathsf{V}$ and the axioms of truth, it would follow that $\neg {\Pr}_{\KT }(\bot)$, contradicting G\"odel's second incompleteness theorem. In conclusion, $\mathsf{UK}^\Kn$, $\mathsf{R}_\DCB$ and $\NEC^\Kn$ ensure that $\KT$ fulfills Target feature $\ref{tgtfeat:DC}$ in a strong sense.
\item $\FS$ meets Target feature \ref{tgtfeat:semantics}, as manifested by the revision semantics, see \cite[ch. 14.1]{Hal14}. This was generalized to modal logic by Stern in \cite{Ste14}. In \S \ref{Sec:Semantics}, I adapt Stern's approach to obtain a revision semantics for multi-agent knowledge. In particular, Theorem \ref{Thm:KT_semantics} establishes that $\KT + \mathsf{UBF} + \mathsf{IA} + \mathsf{In}^+ + \mathsf{In}^-$ can be validated in the revision semantics. The key steps are Lemmata \ref{Lemma:DCB_starting_point} and \ref{Lemma:DCB_stable}, establishing that $\DCB$ (particularly the axiom $\mathsf{R}_\DCB$) can be validated in the revision semantics. 
\item The axioms are formulated for the general setting of a binary knowledge relation $\Kn(\alpha, \phi)$ expressing that $\alpha$ knows $\phi$. Moreover, by $\mathsf{R}_\DCB$, $\KT$ proves that every agent knows the syntactic theorems provable in $\PA$.
\item $\FS$ has the rules $\NEC^\Tr$ and $\CONEC^\Tr$.
\item The subsystem $\UCT$ of $\FS$ is the natural untyped axiomatization of Tarski's compositional semantic definition of truth from \cite{Tar36}. As such, it formalizes philosophically well-motivated assumptions on how truth relates to the atomic formulas and the logical connectives and quantifiers. The remaining rules $\NEC^\Tr$ and $\CONEC^\Tr$ of $\FS$ are also well-motivated, as explained in Target feature \ref{tgtfeat:sym}. By \cite[Corollary 14.24]{Hal14}, $\FS$ proves a philosophically satisfying class of instances of $\mathsf{TB}$. Informally, this is the class of instances of $\Tr(\gquote{\phi}) \leftrightarrow \phi$, where $\phi$ is typable roughly in the sense that it corresponds canonically to a sentence in the language augmenting $\Lang_-$ with a typed hierarchy of truth predicates $\Tr_0, \Tr_1, \cdots$.\footnote{See \cite[Ch. 14]{Hal14} for more details.} Since the axioms of $\FS$ are philosophically well-motivated, this result also explains why these instances of $\mathsf{TB}$ hold.
In the presence of the supplementary axiom $\mathsf{IA}$, this result can be extended to many formulas in which $\Kn$ occurs. A precise general statement and proof of such an extended result would require a considerable technical detour, which is out of scope for the present paper. However, a useful lemma can be conveniently stated and proved:
\begin{lemma}
Let $\phi \in \Lang^0$, and let $S$ be an extension of $\KT + \mathsf{IA}$ (admitting $\NEC^\Kn$). If $S \vdash \Tr(\gquote{\phi}) \leftrightarrow \phi$, then 
$$S \vdash \forall \alpha \in \Ag \dt \Big( \Tr\big( \gquote{\Kn(\udot{\mathrm{num}}(\alpha), \gquote{\phi})} \big) \leftrightarrow \Kn \big( \alpha, \gquote{\phi} \big) \Big).$$
\end{lemma}
\begin{proof}
We work in $\KT + \mathsf{IA}$. Suppose that $\Tr(\gquote{\phi}) \leftrightarrow \phi$. By $\NEC^\Kn$, we have $\forall \alpha \in \Ag \dt \Kn \big( \alpha, \gquote{\Tr(\gquote{\phi}) \leftrightarrow \phi} \big) $. So by $\mathsf{UK}^\Kn$, we have $\forall \alpha \in \Ag \dt \Big( \Kn \big( \alpha, \gquote{\Tr(\gquote{\phi})} \big) \leftrightarrow \Kn \big( \alpha, \gquote{\phi} \big) \Big) $. Now by $\mathsf{IA}$, we obtain $\forall \alpha \in \Ag \dt \Big( \Tr\big( \gquote{\Kn(\udot{\mathrm{num}}(\alpha), \gquote{\phi})} \big) \leftrightarrow \Kn \big( \alpha, \gquote{\phi} \big) \Big)$.
\end{proof}
In particular, if an extension $S$ of $\KT + \mathsf{IA}$ proves the instance of $\mathsf{TB}$ for an $\Lang$-sentence $\phi$ and proves $\Ag(t)$ for an $\Lang$-term $t$, then it proves the instance of $\mathsf{TB}$ for the sentence $\Kn \big( t, \gquote{\phi} \big)$.
\item By Corollary \ref{Cor:KT_consistent}, the supplementary axioms in Figure \ref{Fig:Sup} are jointly consistent with $\KT$. These axioms are relevant to formal epistemology, but I consider them more controversial than the axioms of $\KT$. The axiom $\mathsf{UBF}$ ({\em Untyped Barcan Formula}) asserts that whenever a sentence $\phi(t)$ is known for all $\Lang_\PA$-terms $t$, then $\forall x \dt \phi(x)$ is also known. The axiom $\mathsf{IA}$ ({\em Interaction Axiom}) asserts that something is known to be true iff it is truly known. The axiom $\mathsf{In^+}$ ({\em Positive Introspection}) asserts that whatever is known to be true is known to be known. The axiom $\mathsf{In^-}$ ({\em Negative Introspection}) asserts that whatever is not known to be true is known to be not known. These axioms are also natural for belief. Because of the way that $\KT$ fulfills Target feature \ref{tgtfeat:UT}, $\KT + \mathsf{IA} + \mathsf{In}^+$ also proves the analogous positive introspection principle that if $A$ is known, then $A$ is known to be known, for a large class of sentences $A$ (and analogously for negative introspection). However, by Proposition \ref{Prop:U4_incon}, $\KT$ is inconsistent with the instance of this assertion where $A$ is the sentence $\delta$ from Montague's paradox.

Consider the application obtained by adding the following axioms to $\KT$ (but restricting $\NEC^\Kn$ to proofs in $\KT$), where $s$ and $a$ are constants: $s = \gquote{\underline{0}=\underline{0}}$, $\Ag(a)$, $\neg \Kn(a, \gquote{s = \gquote{\underline{0}=\underline{0}}})$ and $\neg \Kn(a, \gquote{\Tr(s)})$. (This formalizes the possible situation that an agent $a$ does not know ``$s$ is `$0=0$'{}'', and consequently does not know ``$s$ is true''.) Working in this system, note that $\Kn(a, s)$, by $\NEC^\Kn$ (and substitution of identicals), and that $\Tr\big(\gquote{\Kn(a, s)}\big)$, by $\NEC^\Tr$. Moreover, we obtain that $\neg \Kn\big(a, \gquote{\Kn(a, s)}\big)$, by $\mathsf{V}$, $\NEC^\Kn$, $\mathsf{UK}^\Kn$ and $\neg \Kn(a, \gquote{\Tr(s)})$. This example provides a counterexample both to $\mathsf{IA}$ and to the principle that whatever is known is known to be known. Still, $\mathsf{IA}$ is quite a reasonable axiom, under the assumption that all agents know each theorem of the system concerning syntax (in the counterexample, the agent $a$ lacks knowledge about what sentence $s$ refers to); otherwise it needs to be restricted in some way. I have formulated $\mathsf{In}^+$ and $\mathsf{In}^-$ in such a way that they are compatible with this example. Note that the alternative formalization $\forall \alpha \in \Ag \dt \forall \phi \in \udot\Lang^0 \dt \Big( \Tr\big(\udot\Kn(\udot{\mathrm{num}}(\alpha), \udot\gq(\phi)) \big) \rightarrow \Kn\big( \alpha, \udot\Kn(\udot{\mathrm{num}}(\alpha), \udot\gq(\phi)) \big) \Big)$ of positive introspection is not compatible with it.

\item In \S \ref{Sec:DCB_CK_well-behaved}, I show how $\DCB$ meets Target feature \ref{tgtfeat:CK}. The robust self-referential formulation of $\mathsf{R}_\DCB$ plays a key role in these results, as it enables a method of proof based on L\"ob's theorem. My main results on $\DCB$ are that $\DCB$ proves that the defined common belief/knowledge predicates are unique in a certain sense and that they have analogous properties of deductive closure to those of $\Kn$. This concludes the argument that $\KT$ has all the Target features \ref{tgtfeat:paradox}--\ref{tgtfeat:CK}.
\end{enumerate}

The axiom $\Kn^1$-$\Kn^2$ of $\DCB$ ensures that $\Kn^1(y)$ expresses that all agents believe/know $y$. Lastly, note that since $\Pr_\DCB$ represents the provability predicate of $\DCB$, the axiom $\mathsf{R}_\DCB$ entails that $\DCB$ admits the deductive rule $\NEC^\Kn$.\footnote{For any sentence $\phi \in \Lang$, $\DCB \vdash \phi \Longrightarrow \DCB \vdash \Pr_\DCB(\gquote{\phi}) \Longrightarrow \DCB \vdash \forall \alpha \in \Ag \dt \big( \Kn(\alpha, \gquote{\phi}) \big)$.}

\section{$\DCB$ facilitates reasoning about common belief/knowledge}\label{Sec:DCB_CK_well-behaved}

In this section, I show how $\DCB$ self-referentially defines common belief/knowledge and supports expedient reasoning about it. The proofs utilize a technique based on L\"ob's theorem.

Let $A(x)$ be an $\Lang$-formula with only $x$ free. I write $\CK_A(y)$ (read {\em Common Belief/Knowledge for the agents satisfying $A$}) for an $\Lang$-formula with only $y$ free, such that
\begin{equation*}
\DCB \vdash \forall u \in \udot\Lang^0 \dt \Big[ \CK_A(u) \leftrightarrow \forall x \in \Ag \dt \Big(A(x) \rightarrow \big[ \Kn(x, u) \wedge \Kn \big( x, \udot\sbt(\gquote{\CK_A}, \udot{\textrm{gq}}(u), \gquote{y}) \big) \big] \Big) \Big]. \tag{$\textsf{CKE}$}
\end{equation*}
I call this the {\em Common Belief/Knowledge Equivalence}; it follows from the G\"odel--Carnap fixed-point lemma. It is convenient to abbreviate $\textsf{CKE}$ as follows: 
\begin{equation*}
\DCB \vdash \forall u \in \udot{\Lang^0} \dt \big[ \CK_A(u) \leftrightarrow \Psi_A(u, \gquote{\CK_A}) \big],
\end{equation*}
where $\Psi_A(u, v)$ is the formula
\begin{equation*}
\forall x \in \Ag \dt \Big(A(x) \rightarrow \big[ \Kn(x, u) \wedge \Kn \big( x, \udot\sbt(v, \udot{\textrm{gq}}(u), \gquote{y}) \big) \big] \Big).
\end{equation*}

The results in this section are related to results on self-reference due to Smory\'nski \cite[Ch. 4]{Smo85}. They are proved by the same technique, using L\"ob's theorem from \cite{Lob55}: 

\begin{thm}[L\"ob]
Let $S$ be a (recursively enumerable) system extending $\PA$, and let $\phi$ be a sentence in its language. If $S \vdash \Pr_S(\gquote{\phi}) \rightarrow \phi$, then $S \vdash \phi$.
\end{thm}

\begin{thm}\label{Thm:Implied_CK}
Let $A(x), \theta(y), \theta'(y)$ be $\Lang$-formulas such that
\begin{align*}
\DCB  &\vdash \forall u \in \udot{\Lang^0} \dt \big( \theta(u) \rightarrow \Psi_A(u, \gquote{\theta}) \big), \\
\DCB  &\vdash \forall u \in \udot{\Lang^0} \dt \big( \theta'(u) \leftarrow \Psi_A(u, \gquote{\theta'}) \big).
\end{align*}
Then $\DCB  \vdash \forall u \in \udot{\Lang^0} \dt \big( \theta(u) \rightarrow \theta'(u) \big).$
\end{thm}
\begin{proof}
We start by showing that $\DCB $ proves
\begin{align*}
& {\Pr}_{\DCB }\Big( \gquote{\forall y \in \udot{\Lang^0} \dt \big( \theta(y) \rightarrow \theta'(y) \big)} \Big) \rightarrow \\
& \forall u \in \udot{\Lang^0} \dt \big( \Psi_A(u, \gquote{\theta}) \rightarrow \Psi_A(u, \gquote{\theta'}) \big).  \tag{$\dagger$}
\end{align*}
By ${\Pr}_{\DCB }\Big( \gquote{\forall y \in \udot{\Lang^0} \dt \big( \theta(y) \rightarrow \theta'(y) \big)} \Big)$, and universal instantiation internal to ${\Pr}_{\DCB }$, 
\begin{align*}
& \forall u \in \udot{\Lang^0} \dt \Big( {\Pr}_{\DCB } \big( \udot\sbt(\gquote{ \theta(y) \rightarrow \theta'(y)}, \udot{\textrm{gq}}(u), \gquote{y}) \big) \Big).
\end{align*}
Distributing the substitution, and applying $\mathsf{R}_\DCB $ to this, we obtain
\begin{align*}
& \forall u \in \udot{\Lang^0} \dt \forall x \in \Ag \dt \Big( \Kn \big(x, \udot\sbt(\gquote{\theta}, \udot{\textrm{gq}}(u), \gquote{y}) \udot\rightarrow \udot\sbt(\gquote{\theta'}, \udot{\textrm{gq}}(u), \gquote{y}) \big) \Big).
\end{align*}
Now by $\textsf{UK}^\Kn$,
\begin{align*}
& \forall u \in \udot{\Lang^0} \dt \forall x \in \Ag \dt \Big( \Kn \big(x, \udot\sbt(\gquote{\theta}, \udot{\textrm{gq}}(u), \gquote{y}] \big) \rightarrow \Kn \big(x, \udot\sbt(\gquote{\theta'}, \udot{\textrm{gq}}(u), \gquote{y}) \big) \Big).
\end{align*}
It follows from this, and the definition of $\Psi_A$, that
\begin{align*}
& \forall u \in \udot{\Lang^0} \dt \big( \Psi_A(u, \gquote{\theta}) \rightarrow \Psi_A(u, \gquote{\theta'}) \big),
\end{align*}
establishing ($\dagger$).

By ($\dagger$) and the assumption of this theorem,
\begin{align*}
\DCB  \vdash {} & {\Pr}_{\DCB }\Big( \gquote{\forall y \in \udot{\Lang^0} \dt \big( \theta(y) \rightarrow \theta'(y) \big)} \Big) \rightarrow \\
& \forall y \in \udot{\Lang^0} \dt \big( \theta(y) \rightarrow \theta'(y) \big). 
\end{align*}
So by L\"ob's theorem,
\begin{align*}
\DCB  \vdash \forall y \in \udot{\Lang^0} \dt \big( \theta(y) \rightarrow \theta'(y) \big). 
\end{align*}
as desired.
\end{proof}

From the above theorem we obtain uniqueness of the defined common belief/knowledge predicate, with respect to provability in $\DCB $:

\begin{cor}[Uniqueness of defined common knowledge]\label{Cor:Unique_CK}
Let $\CK'_A(y)$ be an $\Lang$-formula such that
\begin{align*}
\DCB  \vdash \forall u \in \udot{\Lang^0} \dt \big[ \CK'_A(u) \leftrightarrow \Psi_A(u, \gquote{\CK'_A}) \big].
\end{align*}
Then $\DCB  \vdash \forall y \in \udot{\Lang^0} \dt \big[ \CK_A(y) \leftrightarrow \CK'_A(y) \big].$
\end{cor}
\begin{proof}
Since
\begin{align*}
\DCB  &\vdash \forall u \in \udot{\Lang^0} \dt \big[ \CK_A(u) \leftrightarrow \Psi_A(u, \gquote{\CK_A}) \big], \\
\DCB  &\vdash \forall u \in \udot{\Lang^0} \dt \big[ \CK'_A(u) \leftrightarrow \Psi_A(u, \gquote{\CK'_A}) \big],
\end{align*}
the result follows by applying Theorem \ref{Thm:Implied_CK} in both directions.
\end{proof}
The philosophical upshot is that if we require generally of defined common knowledge predicates that our theory of syntax $\Base$ (or even $\DCB$) is sufficient to witness that they are adequate representations of common knowledge, then they are unique up to equivalence in $\DCB$ (and consequently in $\KT$ as well).

Here is a useful application of the above corollary:

\begin{lemma}\label{Lem:Conj_CK}
Let $A(x) \in \Lang$ and let $\phi, \psi \in  \Lang^0$.
\begin{align*}
\DCB  \vdash \CK_A(\gquote{\phi}) \wedge \CK_A(\gquote{\psi}) \leftrightarrow \CK_A(\gquote{\phi \wedge \psi}).
\end{align*}
\end{lemma}
\begin{proof}
Let $\theta$ be the formula $\CK_A(\gquote{\phi}) \wedge \CK_A(\gquote{\psi})$. We work in $\DCB $. By Corollary \ref{Cor:Unique_CK}, it suffices to show that 
\begin{align*}
\theta \leftrightarrow \Psi_A(\gquote{\phi \wedge \psi}, \gquote{\theta}).
\end{align*}
This is shown by the following equivalences:
\begin{align*}
& \CK_A(\gquote{\phi}) \wedge \CK_A(\gquote{\psi}) \\
\iff {} & \forall x \dt \Big(A(x) \rightarrow \big[ \Kn(x, \gquote{\phi}) \wedge \Kn(x, \gquote{\psi}) \wedge \Kn \big( x, \gquote{\CK_A(\gquote{\phi})} \big) \wedge \Kn \big( x, \gquote{\CK_A(\gquote{\psi})} \big) \big] \Big) \\
\iff {} & \forall x \dt \Big(A(x) \rightarrow \big[ \Kn(x, \gquote{\phi \wedge \psi}) \wedge \Kn \big( x, \gquote{\CK_A(\gquote{\phi}) \wedge \CK_A(\gquote{\psi})} \big) \big] \Big).
\end{align*}
The last equivalence follows from $\textsf{UK}^\Kn$ and $\NEC^\Kn$.
\end{proof}

Theorem \ref{Thm:Implied_CK} can be generalized as follows:

\begin{thm}\label{Thm:General_CK}
Suppose that $A(x), B(x), \theta(y), \theta'(y)$ are $\Lang$-formulas and $f$ is a function definable in $\DCB $, such that $\DCB $ proves
\begin{align*}
 & \forall u \in \udot{\Lang^0} \dt \big( \theta(u) \rightarrow \Psi_A(u, \gquote{\theta}) \big), \\
 & \forall u \in \udot{\Lang^0} \dt \big( \theta'(u) \leftarrow \Psi_B(u, \gquote{\theta'}) \big), \\
 & \forall x \dt \big( B(x) \rightarrow A(x) \big), \\
 & \forall u \in \udot{\Lang^0} \dt \forall x \in B \dt \big( \Kn(x, u) \rightarrow \Kn(x, f(u)) \big). 
\end{align*}
Then $\DCB \vdash \forall u \in \udot{\Lang^0} \dt  \big( \theta(u) \rightarrow \theta'(f(u)) \big).$
\end{thm}
\begin{proof}
%Let $\Phi_{A,B}(u, u')$ be the formula $\big( \forall x \dt ( B(x) \rightarrow A(x) ) \big) \wedge \big( \forall x \dt (\Kn(x, u) \rightarrow \Kn(x, u')) \big).$ 
It suffices to indicate how the proof of Theorem \ref{Thm:Implied_CK} is amended. The key difference is that now we need to show that $\DCB $ proves
\begin{align*}
& {\Pr}_{\DCB }\Big( \gquote{\forall y \in \udot{\Lang^0} \dt \big( \theta(y) \rightarrow \theta'(f(y)) \big)} \Big) \rightarrow \\
& \forall u \in \udot{\Lang^0} \dt \big( \Psi_A(u, \gquote{\theta}) \rightarrow \Psi_B(f(u), \gquote{\theta'}) \big).  \tag{$\ddagger$}
\end{align*}
Just as before, we prove from ${\Pr}_{\DCB }\Big( \gquote{\forall y \in \udot{\Lang^0} \dt \big( \theta(y) \rightarrow \theta'(f(y)) \big)} \Big)$ that
\begin{align*}
& \forall u \in \udot{\Lang^0} \dt \forall x \dt \Big( \Kn \big(x, \udot\sbt(\gquote{\theta}, \udot{\textrm{gq}}(u), \gquote{y}] \big) \rightarrow \Kn \big(x, \udot\sbt(\gquote{\theta'}, \udot{\textrm{gq}}(f(u)), \gquote{y}) \big) \Big).
\end{align*}
From this and the new assumptions it follows that
\begin{align*}
& \forall u \in \udot{\Lang^0} \dt \big( \Psi_A(u, \gquote{\theta}) \rightarrow \Psi_B(f(u), \gquote{\theta'}) \big),
\end{align*}
establishing ($\ddagger$).
\end{proof}

\begin{cor}\label{Cor:Monotone_CK}
Suppose that $A(x), B(x)$ are $\Lang$-formulas and $f$ is a function definable in $\DCB $, such that $\DCB $ proves
\begin{align*}
 & \forall x \dt \big( B(x) \rightarrow A(x) \big), \\
 & \forall u \in \udot{\Lang^0} \dt \forall x \in B \dt \big( \Kn(x, u) \rightarrow \Kn(x, f(u)) \big). 
\end{align*}
Then $\DCB \vdash \forall u \in \udot{\Lang^0} \dt  \big( \CK_A(u) \rightarrow \CK_B(f(u)) \big).$
\end{cor}
\begin{proof}
A direct application of Theorem \ref{Thm:General_CK} with $\CK_A$ as $\theta$ and $\CK_B$ as $\theta'$.
\end{proof}

The following theorem shows that the deductive closure properties axiomatized for $\Kn$ in $\DCB$ carry over to each common knowledge predicate $\CK_A$, without the need for any primitive predicate or axiom for $\CK_A$:

\begin{thm}[Deductive closure of common knowledge]\label{Thm:CK_main}
Let $A(x) \in \Lang$.
\begin{enumerate}[{\rm (a)}]
\item $\DCB  \vdash \forall \phi, \psi \in \udot{\Lang^0} \dt \big( \CK_A(\phi) \wedge \CK_A(
\phi \udot\rightarrow \psi) \rightarrow \CK_A(\psi) \big)$
\item $\DCB  \vdash \forall \phi \in \udot{\Lang^0} \dt \big( {\Pr}_{\DCB }(\phi) \rightarrow \CK_A(\phi) \big)$
\item For each $\phi \in \Lang^0$: $\DCB \vdash \phi \phantom{~~} \Longrightarrow \phantom{~~} \DCB \vdash \CK_A(\gquote{\phi})$
\end{enumerate}
\end{thm}
\begin{proof}
\begin{enumerate}[{\rm (a)}]
\item $\DCB $ defines a function $f$ such that
\begin{align*}
\DCB  \vdash \forall \phi, \psi \in \udot{\Lang^0} \dt \big( f(\phi \udot\wedge (\phi \udot\rightarrow \psi)) = \psi \big).
\end{align*}
Moreover, 
\begin{align*}
\DCB  \vdash \forall \phi, \psi \in \udot{\Lang^0} \dt \forall x \in \Ag \dt \big( \Kn(x, \phi \udot\wedge (\phi \udot\rightarrow \psi)) \rightarrow \Kn(x, \psi) \big).
\end{align*}
So by Corollary \ref{Cor:Monotone_CK},
\begin{align*}
\DCB  \vdash \CK_A (\phi \udot\wedge (\phi \udot\rightarrow \psi) ) \rightarrow \CK_A(\psi).
\end{align*}
By Lemma \ref{Lem:Conj_CK},
\begin{align*}
\DCB  \vdash \CK_A(\phi) \wedge \CK_A(\phi \udot\rightarrow \psi) \leftrightarrow \CK_A(\phi \udot\wedge (\phi \udot\rightarrow \psi)),
\end{align*}
whence the result follows.
\item By Theorem \ref{Thm:Implied_CK}, with $\Pr_\DCB$ as $\theta$ and with $\CK_A$ as $\theta'$, it suffices to show that
\begin{align*}
\DCB  \vdash \forall u \in \udot{\Lang^0} \dt \big( {\Pr}_{\DCB }(u) \rightarrow \Psi_A(u, \gquote{{\Pr}_{\DCB }(y)}) \big).
\end{align*}
By $\mathsf{R}_\DCB $,
\begin{align*}
\DCB  \vdash \forall u \in \udot{\Lang^0} \dt \big( {\Pr}_{\DCB }(u) \rightarrow \forall x \in \Ag \dt \Kn(x, u) \big).
\end{align*}
It is a well-known property of provability that
\begin{align*}
\DCB  \vdash \forall u \in \udot{\Lang^0} \dt \big[ {\Pr}_{\DCB }(u) \rightarrow {\Pr}_{\DCB }\big( \udot\sbt(\gquote{{\Pr}_{\DCB }(y)}, \udot\gq(u), \gquote{y}) \big) \big].
\end{align*}
So by $\mathsf{R}_\DCB $,
\begin{align*}
\DCB  \vdash \forall u \in \udot{\Lang^0} \dt \big[ {\Pr}_{\DCB }(u) \rightarrow \forall x \in \Ag \dt \Kn \big( x, \udot\sbt(\gquote{{\Pr}_{\DCB }(y)}, \udot\gq(u), \gquote{y}) \big) \big].
\end{align*}
Combining the second and fourth displayed item, we obtain the first displayed item, as desired.
\item Assume that $\DCB \vdash \phi$. Since $\Pr_\DCB$ represents provability in $\DCB$, we have $\Base \vdash \Pr_\DCB(\gquote{\phi})$. Hence, $\DCB \vdash \CK_A(\gquote{\phi})$, by the previous item of this theorem. \qedhere
\end{enumerate}
\end{proof}

The above theorem establishes that $\KT$ proves the generalizations to common knowledge of its axioms $\mathsf{UK}^\Kn$ and  $\mathsf{R}_\DCB$. It remains open whether $\KT$ admits full necessitation for common knowledge, i.e. whether $\KT \vdash \phi \phantom{~~} \Longrightarrow \phantom{~~} \KT \vdash \CK_A(\gquote{\phi})$.

\section{Revision semantics for multi-agent knowledge}\label{Sec:Semantics}

The revision semantics, introduced by Gupta and Herzberger in \cite{Gup82}, \cite{Her82a} and \cite{Her82b}, is an approach to resolving the liar paradox. The basic idea is to start with a ground model $\langle \mathcal{M}, P_0 \rangle$ (where $P_0$ is an interpretation of the truth predicate), and iteratively revise the interpretation of the truth predicate to $P_1, P_2, \cdots$, by the recursion $P_{n+1} =_\df \big\{ \mathrm{gc}(\sigma) \mid \langle \mathcal{M}, P_n \rangle \models \sigma \big\}$. The map defined by $P \mapsto \big\{ \mathrm{gc}(\sigma) \mid \langle \mathcal{M}, P \rangle \models \sigma \big\}$, for all $P \subseteq \mathbb{N}$, is called the {\em revision semantic operator}.

It follows from McGee's theorem \cite{McG85}, that $\FS$ is $\omega$-inconsistent. In particular, it has no standard model. However, there is a connection between $\FS$ and the revision semantics, which vindicates $\FS$: Take the standard model of arithmetic expanded with an arbitrary interpretation of the truth predicate as the ground model; call it $\langle \mathbb{N}, P_0 \rangle$. In \cite{Hal94}, Halbach showed that $\FS$ is {\em locally validated} in the revision semantics, meaning that for each finite fragment of $\FS $, there is $n \in \mathbb{N}$, such that for all $m \geq n$, $\langle \mathbb{N}, P_m \rangle$ is a standard model of that fragment. Since every proof only uses finitely many axioms, one only needs a finite fragment for any particular deductive application of the system. Hence, for any application of $\FS$, an adequate standard model can be obtained by a finite iteration of the revision semantic operator.

In \cite{Ste14}, Stern showed that necessitation and truth (treated as predicates) can be formalized in a range of systems extending $\FS$, and provided a flexible generalized revision semantics locally validating them. Since $\KT$ axiomatizes knowledge for multiple agents, I have generalized this semantics slightly. It builds on a conventional possible worlds semantics; each world is a standard $\Lang_-$-structure satisfying $\Base$, and each agent is associated with a binary accessibility relation on the set of worlds. It involves evaluation functions, which map each world to the set of G\"odel codes of true sentences in that world. Any evaluation function induces interpretations of $\Kn$ and $\Tr$ in each world, expanding each world to an $\Lang$-structure. The revision semantic operator maps each evaluation function to a revised evaluation function, thus inducing a revised $\Lang$-expansion of each world.

Theorem \ref{Thm:KT_semantics} establishes that $\KT + \mathsf{UBF} + \mathsf{IA} + \mathsf{In}^+ + \mathsf{In}^-$ is locally validated in this revision semantics, in the sense that under certain natural assumptions, any finite fragment of $\KT + \mathsf{UBF} + \mathsf{IA} + \mathsf{In}^+ + \mathsf{In}^-$ is eventually satisfied in each world after finitely many revisions. In particular, any finite fragment of $\KT + \mathsf{UBF} + \mathsf{IA} + \mathsf{In}^+ + \mathsf{In}^-$ has a standard model, and $\KT + \mathsf{UBF} + \mathsf{IA} + \mathsf{In}^+ + \mathsf{In}^-$ is consistent. Lemmata \ref{Lemma:DCB_starting_point} and \ref{Lemma:DCB_stable} constitute two key steps in the proof of this theorem, establishing that $\DCB$ can be validated in the revision semantics.

The results of \cite{Ste14} are developed in a setting where the base system is $\PA$ extended with induction for the language obtained by augmenting $\Lang_\PA$ with the predicates $\mathrm{N}$ and $\Tr$. \cite{Ste14} also provides an explanation for how to generalize to other base systems, which applies to $\PA(\Lang)$. I assume that $\Base = \PA(\Lang)$ throughout this section. 

I proceed to go through the generalized definitions and theorems that are relevant for this paper. All worlds are assumed to be expansions of the standard model $\mathbb{N}$ and, for simplicity, it is assumed that the interpretation of $\Ag$ is identical in all worlds. It is straight-forward to generalize to diverse interpretations of $\Ag$. Fix a non-empty subset $G \subseteq \mathbb{N}$, whose elements are called {\em agents}. For any first-order language $L$ expanding $\Lang_\PA$, an $L$-{\em agency-frame for $G$} is a tuple $F = \langle W, (R_\alpha)_{\alpha \in G} \rangle$, where $W$ is a set of $L$-expansions of $\mathbb{N}$ interpreting $\Ag$ by $G$, and $R_\alpha$ is a binary relation on $W$, for each $\alpha \in G$. 

For the rest of this section, fix an $\Lang_-$-agency-frame $F$ for $G$. For each agent $\alpha \in G$: $\langle W, R_\alpha \rangle$ is called the $\Lang_-$-{\em Kripke-frame of $\alpha$}, the elements of $W$ are called {\em worlds}, and $R_\alpha$ is called the {\em accessibility relation of $\alpha$}. An {\em evaluation function of $F$} is a function $f : W \rightarrow \mathcal{P}(\mathbb{N})$, thought of as mapping a world to the set of G\"odel codes of true sentences in that world. $\mathrm{Val}_F$ denotes the set of all evaluation functions of $F$. For the rest of this paragraph, fix $f \in \mathrm{Val}_F$. $f_\alpha : W \rightarrow \mathcal{P}(\mathbb{N})$ and $f^\Kn : W \rightarrow \mathcal{P}(\mathbb{N} \times \mathbb{N})$ are defined as follows, for all $\alpha \in G$ and all $w \in W$:
\begin{align*}
f_\alpha(w) &=_\df \bigcap \{ f(v) \mid w R_\alpha v \} \\
f^\Kn(w) &=_\df \{ \langle \alpha, \phi \rangle \mid \alpha \in G \wedge \phi \in f_\alpha(w) \}
\end{align*}
For each $\Lang_-$-model $w \in W$, $f$ induces an $\Lang$-expansion $w^f$ of $w$, in which $\Tr$ is interpreted by $f(w)$, $\Kn^2$ is interpreted by $f^\Kn(w)$, and $\Kn^1$ is interpreted by $\big\{\phi \mid \forall \alpha \in G \dt (\langle \alpha, \phi \rangle \in f^\Kn(w))\big\}$. Thus, $f$ also induces an $\Lang$-agency-frame $F^f =_\df \big\langle W^f, (R^f_\alpha)_{\alpha \in G} \big\rangle$, where $W^f =_\df \big\{w^f \mid w \in W\big\}$ and $R^f_\alpha =_\df \big\{ \big\langle v^f, w^f \big\rangle \mid \big\langle v, w \big\rangle \in R_\alpha \big\}$, for each $\alpha \in G$. 

The {\em revision semantic operator}, $\Gamma : \mathrm{Val}_F \rightarrow \mathrm{Val}_F$, is defined as follows, for all $f \in \mathrm{Val}$ and all $w \in W$:
\begin{align*}
\big(\Gamma(f)\big)(w) =_\df \big\{ \mathrm{gc}(\phi) \mid w^f \models \phi \big\}
\end{align*}
Note that for any $\Lang$-agency-frame of the form $F^f$ (where $F$ is an $\Lang_-$-agency-frame and $f \in \mathrm{Val}_F$), $\Gamma$ induces a revised $\Lang$-agency-frame $F^{\Gamma(f)}$. This revision is iterated by iterated applications of $\Gamma$, providing a revision semantics appropriate for the setting of this paper. This revision semantics is closely related to the system $\mathsf{BEFS}$ ({\em Basic Epistemic Friedman--Sheard}), axiomatized in Figure \ref{Fig:BEFS}.\footnote{$\mathsf{BEFS}$ is analogous to Stern's system $\mathsf{BMFS}$ from \cite{Ste14}.}

The axiom $\mathsf{UNS}^\Kn$ ({\em Untyped Necessitated Substitution principle}) of $\mathsf{BEFS}$ asserts that knowledge is preserved under substitution of $\Lang_\PA$-terms that evaluate to the same value. The axiom $\mathsf{UND}^\Kn$ ({\em Untyped Necessity of Distinctness}) of $\mathsf{BEFS}$ asserts that it is known whenever two $\Lang_\PA$-terms evaluate to distinct values (the analogue for equality is provable in the system). 

\begin{prop}\label{Prop:KT_proves_BEFS}
$\KT + \mathsf{UBF} + \mathsf{IA}  \equiv \mathsf{BEFS} + \mathsf{V} + \mathsf{R}_\DCB$.
\end{prop}
\begin{proof}
By $\CONEC^\Tr$ and $\NEC^\Kn$, $\KT + \mathsf{UBF} + \mathsf{IA}$ admits the rule $\Tr/\Kn$. By $\NEC^\Tr$ and $\Tr/\Kn$, $\mathsf{BEFS} + \mathsf{V} + \mathsf{R}_\DCB$ admits $\NEC^\Kn$. It remains only to prove $\mathsf{UNS}^\Kn$ and $\mathsf{UND}^\Kn$ in $\KT$. Since evaluation of $\Lang_\PA$-terms is primitive recursive, $\PA$ proves that the antecedent $s^\circ = t^\circ$ of $\mathsf{UNS}^\Kn$ is equivalent to $\Pr_\PA(s \udot= t)$. Therefore, it is straight-forward to prove $\mathsf{UNS}^\Kn$ from $\mathsf{R}_\DCB$, $\mathsf{UK}^\Kn$ and $\NEC^\Kn$. The same method also works for $s^\circ \neq t^\circ$, establishing $\mathsf{UND}^\Kn$ analogously.
\end{proof}

One might feel that the untyped version of axiom $\mathsf{4}$ from modal logic,
\begin{align*}
\forall \alpha \in \Ag \dt \forall \phi \in \udot\Lang^0 \dt \Big(& \Kn(\alpha, \phi) \rightarrow \Kn\big( \alpha, \udot\Kn(\udot{\mathrm{num}}(\alpha), \udot\gq(\phi)) \big) \Big), \tag{$\mathsf{U4}$}
\end{align*}
is a straight-forward formalization of the positive introspection principle, but as established by the following proposition, it is inconsistent with $\KT$. It turns out that $\mathsf{In}^+$ is the appropriate formalization of positive introspection in this genuinely untyped setting.

\begin{prop}\label{Prop:U4_incon}
$\KT + \mathsf{U4}$ is inconsistent.
\end{prop}
\begin{proof}
Consider $\delta$, such that $\Base \vdash \delta \leftrightarrow \neg \Kn(\gquote{\delta})$. Assume $\Kn(\gquote{\delta})$. Then we obtain $\Kn \big(\gquote{\Kn(\gquote{\delta})} \big)$ from $\mathsf{U4}$. But $\Base \vdash \neg \delta \leftrightarrow \Kn(\gquote{\delta})$, so by basic deductive closure (following from $\NEC^\Kn$ and $\mathsf{UK}^\Kn$), we obtain $\Kn(\gquote{\neg \delta})$ and $\Kn(\bot)$. Now by $\mathsf{V}$, we have $\Tr(\bot)$, from which $\bot$ is obtained by the axiom $\UCT^\Atom$ of $\FS$. This proves $\neg \Kn(\gquote{\delta})$, which is equivalent to $\delta$. Having proved $\delta$, we obtain $\Kn(\delta)$ by $\NEC^\Kn$, which is equivalent to $\neg \delta$, a contradiction.
\end{proof}

\begin{figure}

\caption{Axioms and rules of $\mathsf{BEFS}$}
\label{Fig:BEFS}

\begin{center}
 \vspace{12pt}
{\bf The system $\mathsf{BEFS}$}

\vspace{12pt}
$\mathsf{BEFS}$ extends $\FS $ with the following axioms and rules:

\vspace{12pt}
{\bf Axioms of $\mathsf{BEFS}$}
\end{center}
\begin{align*}
\exists \alpha \dt \big(& \Ag(\alpha) \big) \tag{$\textsf{Non-triviality}$} \\  
\forall \phi \in \udot\Lang^{{0}} \dt \big(& \Kn^1(\phi) \leftrightarrow \forall \alpha \in \Ag \dt \Kn^2(\alpha, \phi) \big) \tag{$\Kn^1$-$\Kn^2$} \\
\forall \alpha \in \Ag \dt \forall \phi, \psi \in \udot \Lang^{{0}} \dt \Big(& \big( \Kn(\alpha, \phi) \wedge \Kn(\alpha, \phi \udot\rightarrow \psi) \big)  \rightarrow \Kn(\alpha, \psi) \Big) \tag{$\textsf{UK}^\Kn$} \\
\forall \alpha \in \Ag \dt \forall v \in \udot \Var \dt \forall s, t \in \udot \Term_{\Lang_\PA}^0 \dt \forall \phi \in \udot \Lang^1[v] \dt \Big(& s^\circ = t^\circ \rightarrow \Kn\big(\alpha, \udot\sbt(\phi, s, v)) \leftrightarrow \Kn(\alpha, \udot\sbt(\phi, t, v) \big) \Big) \tag{$\mathsf{UNS}^\Kn$} \\
\forall \alpha \in \Ag \dt \forall s, t \in \udot\Term_{\udot\Lang_\PA}^0 \dt \Big(& s^\circ \neq t^\circ \rightarrow \big( \Kn(\alpha, s \udot\neq t) \big) \Big) \tag{$\mathsf{UND}^\Kn$} \\
\forall \alpha \in \Ag \dt \forall v \in \udot\Var \dt \forall \phi \in \udot\Lang^1[v] \dt \Big(& \big( \forall t \in \udot\Term_{\udot\Lang_\PA}^0 \Kn(\alpha, \udot\sbt(\phi, t, v)) \big) \rightarrow \Kn(\alpha, \udot\forall v \phi) \Big) \tag{$\mathsf{UBF}^\Kn$} \\
\forall \alpha \in \Ag \dt \forall \phi \in \udot\Lang^0 \dt \Big(& \Kn\big(\alpha, \udot\Tr(\udot\gq(\phi)) \big) \leftrightarrow \Tr\big( \udot\Kn(\udot{\mathrm{num}}(\alpha), \udot\gq(\phi)) \big) \Big) \tag{$\mathsf{IA}$} 
\end{align*}

\vspace{6pt}
\begin{center} 
{\bf Rules of $\mathsf{BEFS}$}
\end{center}
\begin{align*}
\vdash \Tr(\gquote{\phi}) \phantom{~~} &\Longrightarrow \phantom{~~} \vdash \forall \alpha \in \Ag \dt \Kn(\alpha, \gquote{\phi}) \textrm{, for each $\phi \in \Lang^0$.}  \tag{$\Tr/\Kn$} 
\end{align*}

\end{figure}

A {\em pseudo-system} is defined as a recursively enumerable set of sentences in a first-order language.  Pseudo-systems constitute an auxiliary tool for approximating systems. For each $2 \leq n < \omega$, the pseudo-systems $\mathsf{BEFS}_n$ and $\mathsf{EFS}_n$ are defined as $\mathsf{BEFS}$ and $\mathsf{EFS}$, respectively, except that at most a total of $n - 1$ applications of $\NEC^\Tr$ and $\CONEC^\Tr$ are allowed in a proof. The systems $\mathsf{BEFS}_0$ and $\mathsf{EFS}_0$ are defined as $\PA(\Lang)$. The systems $\mathsf{BEFS}_1$ and $\mathsf{EFS}_1$ are given by the axioms of $\mathsf{BEFS}$ and $\mathsf{EFS}$, respectively, in addition to the axiom $\mathsf{R}_{\PA(\Lang)}^\Tr$ below, and the deductive rule $\Tr / \Kn$.
\begin{align*}
\forall \phi \in \Lang^0 \dt \big( {\Pr}_{\PA(\Lang)}(\phi) \rightarrow \Tr(\phi) \big) \tag{$\mathsf{R}_{\PA(\Lang)}^\Tr$}
\end{align*}

For any function $h : A \rightarrow B$ and any $A' \subseteq A$, $h[A'] =_\df \{h(a) \mid a \in A'\}$. Let $R \subseteq A \times A$. $R$ is {\em Euclidean} if $\forall a, b, c \in A \dt \big( (aRb \wedge aRc) \rightarrow bRc \big)$.
\begin{thm}\label{Thm:RevisionBEFS}
Let $f \in \mathrm{Val}_F$ and let $n < \omega$.
\begin{enumerate}[{\rm (a)}]
\item Then
\begin{align*}
f \in \Gamma^n[\mathrm{Val}_F] &\Longleftrightarrow \forall w \in W \dt \big( w^f \models \mathsf{BEFS}_n \big). 
\end{align*}
\item If $R_\alpha$ is reflexive and $1 \leq n$, then
\begin{align*}
f \in \Gamma^n[\mathrm{Val}_F] &\Longleftrightarrow \forall w \in W \dt \big( w^f \models \mathsf{BEFS}_n + \mathsf{V} \big).
\end{align*}
\item If $R_\alpha$ is transitive and $1 \leq n$, then
\begin{align*}
f \in \Gamma^n[\mathrm{Val}_F] &\Longleftrightarrow \forall w \in W \dt \big( w^f \models \mathsf{BEFS}_n + \mathsf{In}^+ \big). 
\end{align*}
\item If $R_\alpha$ is Euclidean and $1 \leq n$, then
\begin{align*}
f \in \Gamma^n[\mathrm{Val}_F] &\Longleftrightarrow \forall w \in W \dt \big( w^f \models \mathsf{BEFS}_n + \mathsf{In}^- \big). 
\end{align*}
\end{enumerate}
\end{thm}
\begin{proof}
I explain how to generalize the proofs of Theorems 4.11 and 4.12 in \cite{Ste14} to the multi-agent setting. First note that the axioms and rules of $BMFS$ from \cite{Ste14} are directly generalized to corresponding multi-agent axioms of $\mathsf{BEFS}$. Moreover, $\mathsf{V}$ is the direct generalization of $T'$ from \cite{Ste14} to the multi-agent setting. Using $\mathsf{IA}$ and $\mathsf{UCT}^\neg$, it is seen that the axioms $\mathsf{In}^+$ and $\mathsf{In}^-$ are equivalent to the direct generalizations of $4'$ and $E'$ in \cite{Ste14} to the multi-agent setting, respectively. 

$\textsf{Non-triviality}$ follows from that $G$ is non-empty. $\Kn^1$-$\Kn^2$ follows from the interpretation of $\Kn^1$ in the definition of $w^f$ (for any $w \in W$). The other axioms concerned are all universally quantified over $\Ag$. We generalize from Stern's single unary necessity predicate $\mathrm{N}$ to a countable set $\{\Kn_\alpha \mid \alpha \in G\}$ of unary knowledge predicates (one for each agent). Note that for a standard model, interpretations of $\Ag$ and the binary knowledge predicate $\Kn$ induce a set $G$ of agents and an interpretation of $\Kn_\alpha$ for each $\alpha \in G$; and vice versa. Letting $\alpha \in G$ be arbitrary, the verification of the axioms then proceeds exactly as in \cite{Ste14}, for the single unary predicate $\Kn_\alpha$ and the single accessibility relation $R_\alpha$.
\end{proof}

The left-to-right directions of the above theorem show that for any finite fragment of $\mathsf{BEFS} + \mathsf{V} + \mathsf{In}^+ + \mathsf{In}^-$, it is satisfied by all the worlds in the $\Lang$-agency-frame induced by an evaluation function of $F$ obtained by a finite number of revisions of an arbitrary evaluation function of $F$. The following lemmata are provided for the sake of extending this result to the axiom $\mathsf{R}_\DCB$.

Let $R \subseteq A \times A$. $R$ is {\em left-total} if for each $a \in A$ there is $b \in A$, such that $a R b$. Note that if $R$ is reflexive, then it is left-total.

\begin{lemma}\label{Lemma:DCB_starting_point}
If for all $\alpha \in G$, $R_\alpha$ is left-total, and $f$ is an evaluation function, such that for all $w \in W$, there is an $\Lang$-system $S_w$ extending $\DCB$, such that $f(w) = \big\{ \mathrm{gc}(\sigma) \mid S_w \vdash \sigma \big\}$, then for all $w \in W$, $w^f \models \DCB$. 
\end{lemma}
\begin{proof}
Let $w \in W$ be arbitrary. Since the theorems of $S_w$ are closed under modus ponens and $R_\alpha$ is left-total for all $\alpha \in G$, we have that $w^f \models \mathsf{UK}^\Kn$. Since $w$ is standard, we have for all $\sigma \in \Lang^0$ that $S_w \vdash \sigma \iff w \models \Pr_{S_w}(\gquote{\sigma})$. So since $S_w \vdash \DCB$, $f(w) = \big\{ \mathrm{gc}(\sigma) \mid S_w \vdash \sigma \big\}$, and $R_\alpha$ is left-total for all $\alpha \in G$, we have that $w^f \models \mathsf{R}_\DCB$.
\end{proof}

\begin{lemma}\label{Lemma:DCB_stable}
Let $f \in \mathrm{Val}_F$, and assume that for all $\alpha \in G$, $R_\alpha$ is left-total, and that for all $w \in W$, $w^f \models \DCB$. Then for all $n < \omega$, and for all $w \in W$, $w^{\Gamma^n(f)} \models \DCB$.
\end{lemma}
\begin{proof}
We prove this by induction. Let $w \in W$ be arbitrary. By assumption, $w^f \models \DCB$. Suppose that $w^{\Gamma^n(f)} \models \DCB$, for some $n < \omega$. Then $\big\{ \mathrm{gc}(\sigma) \mid \DCB \vdash \sigma \big\} \subseteq \big(\Gamma^{n+1}(f)\big)(w)$. Thus, for each $\alpha \in G$, it follows from left-totality of $R_\alpha$ that $\big\{ \big\langle \alpha, \mathrm{gc}(\sigma) \big\rangle \mid \DCB \vdash \sigma \big\} \subseteq \big(\Gamma^{n+1}(f)\big)^\Kn(w)$. Moreover, since $w$ is standard, we have for all $\sigma \in \Lang^0$ that $\DCB \vdash \sigma \iff w \models \Pr_\DCB(\gquote{\sigma})$. Combining these facts, we get that $w^{\Gamma^{n+1}(f)} \models \mathsf{R}_\DCB$. Moreover, it follows from Theorem \ref{Thm:RevisionBEFS} that for all $k \geq 1$, $w^{\Gamma^k(f)} \models \Kn^1\text{-}\Kn^2 + \mathsf{UK}^\Kn$. So $w^{\Gamma^{n+1}(f)} \models \DCB$, as desired.
\end{proof}

\begin{thm}\label{Thm:KT_semantics}
Let $f$ be an evaluation function, such that for all $w \in W$, there is an $\Lang$-system $S_w \vdash
\DCB$, such that $f(w) = \big\{ \mathrm{gc}(\sigma) \mid S_w \vdash \sigma \big\}$. 
\begin{enumerate}[{\rm (a)}]
\item If $R_\alpha$ is reflexive for all $\alpha \in G$, then for each finite $\Phi \subseteq \KT + \mathsf{UBF} + \mathsf{IA}$, there is $m < \omega$, such that for all $n \geq m$ and all $w \in W$, 
\begin{align*}
w^{\Gamma^n(f)} \models \Phi .
\end{align*}
\item If $R_\alpha$ is reflexive and transitive for all $\alpha \in G$, then for each finite $\Phi \subseteq \KT + \mathsf{UBF} + \mathsf{IA} + \mathsf{In}^+$, there is $m < \omega$, such that for all $n \geq m$ and all $w \in W$, 
\begin{align*}
w^{\Gamma^n(f)} \models \Phi .
\end{align*}
\item If $R_\alpha$ is reflexive, transitive and Euclidean for all $\alpha \in G$, then for each finite $\Phi \subseteq \KT + \mathsf{UBF} + \mathsf{IA} + \mathsf{In}^+ + \mathsf{In}^-$, there is $m < \omega$, such that for all $n \geq m$ and all $w \in W$, 
\begin{align*}
w^{\Gamma^n(f)} \models \Phi .
\end{align*}
\end{enumerate}
\end{thm}
\begin{proof}
This result is immediate from combining Theorem \ref{Thm:RevisionBEFS}, Lemma \ref{Lemma:DCB_starting_point} and Lemma \ref{Lemma:DCB_stable}.
\end{proof}

\begin{cor}\label{Cor:KT_consistent}
$\KT + \mathsf{UBF} + \mathsf{IA} + \mathsf{In}^+ + \mathsf{In}^-$ is consistent.
\end{cor}
\begin{proof}
This follows from Theorem \ref{Thm:KT_semantics} and the soundness of first-order logic.
\end{proof}

\section{Conclusion and suggestions for further research}

I have exhibited a genuinely untyped solution to the knower paradoxes, which has several distinct advantages. Notably, the framework accommodates a definition of common belief/knowledge by a G\"odel fixed-point formula. In particular, I showed that reasonable axioms of deductive closure for belief suffice to establish that common belief/knowlegde has these deductive closure properties. Thus, an explanation of common knowledge in terms of ordinary knowledge is obtained.

Stern's revision semantics for necessity was generalized to multi-agent knowledge. This semantics generalizes both the traditional Hintikka-style possible worlds semantics for knowledge, and the revision semantics for truth. It turns out that this semantics validates the axiom $\mathsf{IA}$, and is thus suited to applications where it is assumed that all agents know each theorem of the system concerning syntax. An avenue for further research is to investigate whether Stern's semantics can be further generalized to accommodate failures of $\mathsf{IA}$ (see the explanation for how $\KT$ meets Target feature \ref{tgtfeat:Introspec} in \S \ref{Sec:Theory_KT}).

Another avenue for further research is to look for more applications of the techniques utilized in \S \ref{Sec:DCB_CK_well-behaved}. In particular, are there other useful definitions to be made as G\"odel fixed-points, and are there positive applications of L\"ob's theorem to be obtained concerning them?

\section{Acknowledgments}

I would like to extend my appreciation to Rasmus Blanck, Volker Halbach, Graham Leigh, and {\O}ystein Linnebo for their insights, guidance and constructive feedback that greatly contributed to the development of this paper. This research was supported by the Knut and Alice Wallenberg Foundation (KAW) [2015.0179], and by the Swedish Research Council (VR) [2020-00613].

\end{document}